\documentclass[12pt]{amsart}

\usepackage{amsfonts,amsmath,amssymb,latexsym,amsthm,newlfont,enumerate,color}

\newif\ifslide

\theoremstyle{plain}

\ifdefined\spacer
\newtheorem{theorem}{Theorem}
 \else
\newtheorem{theorem}{Theorem}[section]
\fi

\newtheorem{corollary}[theorem]{Corollary}
\newtheorem{lemma}[theorem]{Lemma}

\newtheorem{proposition}[theorem]{Proposition}
\newtheorem{definition-lemma}[theorem]{Definition-Lemma}
\newtheorem{question}[theorem]{Question}
\newtheorem{red-question}[theorem]{\textcolor{red}{Question}}

\newtheorem{conjecture}[theorem]{Conjecture}

\theoremstyle{definition}
\newtheorem{definition}[theorem]{Definition}
\newtheorem{remark}[theorem]{Remark}

\newtheorem{example}[theorem]{Example}

\DeclareMathOperator{\can}{can}

\def\ideal#1.{I_{#1}}
\def\ring#1.{\mathcal {O}_{#1}}

\def\fring#1.{\hat{\mathcal {O}}_{#1}}
\def\proj#1.{\mathbb {P}(#1)}
\def\pr #1.{\mathbb {P}^{#1}}
\def\dpr #1.{\hat{\mathbb {P}}^{#1}}
\def\af #1.{\mathbb A^{#1}}
\def\Hz #1.{\mathbb F_{#1}}
\def\Hbz #1.{\overline{\mathbb F}_{#1}}
\def\fb#1.{\underset #1 {\times}}
\def\rest#1.{\underset {\ \ring #1.} \to \otimes}
\def\au#1.{\operatorname {Aut}\,(#1)}
\def\deg#1.{\operatorname {deg } (#1)}

\def\pic#1.{\operatorname {Pic}\,(#1)}
\def\pico#1.{\operatorname{Pic}^0(#1)}
\def\picg#1.{\operatorname {Pic}^G(#1)}
\def\ner#1.{NS (#1)}
\def\rdown#1.{\llcorner#1\lrcorner}
\def\rfdown#1.{\lfloor{#1}\rfloor}
\def\rup#1.{\ulcorner{#1}\urcorner}
\def\rcup#1.{\lceil{#1}\rceil}

\def\n1#1.{\operatorname {N_1}(#1)}  
\def\cn1#1.{\overline{\operatorname {N^1}(#1)}} 
\def\cone#1.{\operatorname {NE}(#1)}     
\def\ccone#1.{\overline{\operatorname {NE}}(#1)}
\def\none#1.{\operatorname {NF}(#1)}
\def\cnone#1.{\overline{\operatorname {NF}}(#1)}
\def\mone#1.{\operatorname {NM}(#1)} 
\def\cmone#1.{\overline{\operatorname {NM}}(#1)}

\def\coef#1.{\frac{(#1-1)}{#1}}
\def\vit#1.{D_{\langle #1 \rangle}}
\def\mm#1.{\overline {M}_{0,#1}}
\def\H1#1.{H^1(#1,{\ring #1.})}
\def\ac#1.{\overline {\mathbb F}_{#1}}

\def\adj#1.{\frac {#1-1}{#1}}
\def\spn#1.{\overline{#1}}
\def\pek#1.#2.{\Cal P^{#1}(#2)}
\def\plk#1.#2.{\Cal P^{\leq #1}(#2)}
\def\ev#1.{\operatorname{ev_{#1}}}
\def\ilist#1.{{#1}_1,{#1}_2,\dots}
\def\bminv#1.{(\nu_1,s_1;\nu_2,s_2;\dots ;\nu_{#1},s_{#1};\nu_{r+1})}
\def\zinv#1.{(\nu_1,s_1;\nu_2,s_2;\dots ;\nu_{#1},s_{#1};0)}
\def\iinv#1.{(\nu_1,s_1;\nu_2,s_2;\dots ;\nu_{#1},s_{#1};\infty)}

\def\scr #1.{\mathcal #1}

\def\llist#1.#2.{{#1}_1,{#1}_2,\dots,{#1}_{#2}}
\def\ulist#1.#2.{{#1}^1,{#1}^2,\dots,{#1}^{#2}}
\def\lomitlist#1.#2.{{#1}_1,{#1}_2,\dots,\hat {{#1}_i}, \dots, {#1}_{#2}}
\def\lomitlistz#1.#2.{{#1}_0,{#1}_1,\dots,\hat {{#1}_i}, \dots, {#1}_{#2}}
\def\loc#1.#2.{\Cal O_{#1,#2}}
\def\fderiv#1.#2.{\frac {\partial #1}{\partial #2}}
\def\deriv#1.#2.{\frac {d #1}{d #2}}

\def\map#1.#2.{#1 \longrightarrow #2}
\def\rmap#1.#2.{#1 \dasharrow #2}
\def\emb#1.#2.{#1 \hookrightarrow #2}
\def\non#1.#2.{\text {Spec }#1[\epsilon]/(\epsilon)^{#2}}
\def\Hi#1.#2.{\text {Hilb}^{#1}(#2)}
\def\sym#1.#2.{\operatorname {Sym}^{#1}(#2)}
\def\Hb#1.#2.{\text {Hilb}_{#1}(#2)}
\def\Hm#1.#2.{\Hom_{#1}(#2)}
\def\prd#1.#2.{{#1}_1\cdot {#1}_2\cdots {#1}_{#2}}
\def\Bl #1.#2.{\operatorname {Bl}_{#1}#2}
\def\pl #1.#2.{#1^{\otimes #2}}
\def\mgn#1.#2.{\overline {M}_{#1,#2}}
\def\ialist#1.#2.{{#1}_1 #2 {#1}_2, #2\dots}
\def\pair#1.#2.{\langle #1, #2\rangle}
\def\vandermonde#1.#2.{\left|
\begin{matrix}
1 & 1 & 1 & \dots & 1\\
{#1}_1 & {#1}_2 & {#1}_3 & \dots & {#1}_{#2}\\
{#1}_1^2 & {#1}_2^2 & {#1}_3^2 & \dots & {#1}_{#2}^2\\
\vdots & \vdots & \vdots & \ddots & \vdots\\
{#1}_1^{#2-1} & {#1}_2^{#2-1} & {#1}_2^{#2-1} & \dots & {#1}_{#2}^{#2-1}\\
\end{matrix}
\right|
}
\def\vandermondet#1.#2.{\left|
\begin{matrix}
1 & {#1}_1   & {#1}_1^2 & \dots & {#1}_1^{#2-1}\\
1 & {#1}_2   & {#1}_2^2 & \dots & {#1}_2^{#2-1}\\
1 & {#1}_3   & {#1}_3^2 & \dots & {#1}_3^{#2-1}\\
\vdots & \vdots & \vdots & \ddots & \vdots\\
1 & {#1}_{#2}& {#1}_{#2}^2 & \dots & {#1}_{#2}^{#2-1}\\
\end{matrix}
\right|
}
\def\gr#1.#2.{\mathbb{G}(#1,#2)}

\def\alist#1.#2.#3.{{#1}_1 #2 {#1}_2 #2\dots #2 {#1}_{#3}}
\def\zlist#1.#2.#3.{#1_0 #2 #1_1 #2\dots #2 #1_{#3}}
\def\lomitlist30#1.#2.#3.{{#1}_0,{#1}_1 #2 \dots #2\hat {{#1}_i} #2\dots #2 {#1}_{#3}}
\def\lmap#1.#2.#3.{#1 \overset{#2}{\longrightarrow} #3}
\def\mes#1.#2.#3.{#1 \longrightarrow #2 \longrightarrow #3}
\def\ses#1.#2.#3.{0\longrightarrow #1 \longrightarrow #2 \longrightarrow #3 \longrightarrow 0}
\def\les#1.#2.#3.{0\longrightarrow #1 \longrightarrow #2 \longrightarrow #3}
\def\res#1.#2.#3.{#1 \longrightarrow #2 \longrightarrow #3\longrightarrow 0}
\def\Hi#1.#2.#3.{\text {Hilb}^{#1}_{#2}(#3)}
\def\ten#1.#2.#3.{#1\underset {#2}{\otimes} #3}
\def\lomitlist30#1.#2.#3.{{#1}_0 #2 {#1}_1 #2 \dots #2 \hat {{#1}_i} #2 \dots #2 {#1}_{#3}}
\def\mderiv#1.#2.#3.{\frac {d^{#3} #1}{d #2^{#3}}}

\def\Hom{\operatorname{Hom}}

\def\Proj{\operatorname{Proj}}

\def\Supp{\operatorname{Supp}}
\def\Bs{\operatorname{\mathbf B}}
\def\Exc{\operatorname{Exc}}
\def\Eff{\operatorname{Eff}}

\def\deg{\operatorname{deg}}

\def\Bir{\operatorname{Bir}}

\def\det{\operatorname{det}}

\def\Div{\operatorname{Div}}

\def\mult{\operatorname{mult}}

\def\mov{\operatorname{Mov}}

\def\rest{\operatorname{res}}

\def\bs{\operatorname{Bs}}

\def\p{\mathbb P}

\def\e{\Cal E}

\def\e1{E_1}
\def\e2{E_2}

\def\mapdown#1{\big\downarrow\rlap{$\vcenter{\hbox{$\scriptstyle#1$}}$}}

\def\mapse#1{
{\vcenter{\hbox{$\mathop{\smash{\raise1pt\hbox{$\diagdown$}\!\lower7pt
\hbox{$\searrow$}}\vphantom{p}}\limits_{#1}\vphantom{\mapdown{}}$}}}}

\def\VR#1.{height#1pt&\omit&&\omit&&\omit&&\omit&&\omit&\cr}

\def\VRT#1.{height#1pt&\omit&&\omit&\cr}

\title {On the number of minimal models of a log smooth threefold}
\date{\today}

\author{Paolo Cascini}
\address{Department of Mathematics\\
Imperial College London\\
180 Queen's Gate\\
London SW7 2AZ, UK}
\email{p.cascini@imperial.ac.uk}

\author{Vladimir Lazi\'c}
\address{Mathematisches Institut\\
Universit\"at Bayreuth\\
95440 Bayreuth\\
Germany}
\email{vladimir.lazic@uni-bayreuth.de}

\thanks{The first author was partially supported by an EPSRC Grant. The second author was supported by the DFG-Forschergruppe 790 ``Classification of Algebraic Surfaces and Compact Complex Manifolds". We would like to thank T.~Dorsch, A.-S.~Kaloghiros, Th.~Peternell and C.~Xu for useful discussions.}

\begin{document}

\begin{abstract}
We give a topological bound on the number of minimal models of a class of three dimensional log smooth pairs of general type. 
\end{abstract}

\maketitle
\tableofcontents

\section{Introduction}

The aim of this paper is to prove that the number of log terminal models of certain $3$-dimensional complex terminal projective log smooth pairs is completely determined by their underlying topology as complex manifolds. 

The Minimal Model Program predicts that a complex projective manifold $X$ which is not uniruled has a representative in its birational class with exceptional properties. In other words, if $K_X$ is pseudoeffective, we can associate to $X$ a minimal model -- a variety $Y$ birational to $X$ which admits a nef canonical divisor. Even though minimal models of smooth projective threefolds are in general not smooth nor unique, their singularities are classified \cite{Mori85,Reid85} and moreover, the first author and D.-Q. Zhang \cite{CZ12} provided a topological bound on their singularities. 

It is an important and long-standing conjecture that the number of minimal models of a smooth projective variety is finite up to isomorphism. This is known for projective varieties of general type \cite{BCHM10}, and for threefolds with positive Kodaira dimension \cite{Kawamata97a}. Furthermore, Shokurov's log geography \cite{CS11} (which uses the full Minimal Model Program) and the relative version of the Cone conjecture of Morrison and Kawamata \cite{Kawamata97a} 
imply the conjecture in general, see Theorem \ref{t_finiteness}.

This paper represents the first attempt to bound the number of minimal models of a given log smooth pair of dimension $3$ with respect to the underlying topology as a complex manifold. Our main result is the following.

\begin{theorem}\label{t_main}
Let $p$ and $\rho$ be positive integers, and let $\varepsilon$ be a positive rational number. Let $(X,\sum_{i=1}^p S_i)$ be a $3$-dimensional log smooth pair such that:
\begin{enumerate}
\item[(i)] $X$ is not uniruled,
\item[(ii)] $S_1,\dots,S_p$ are distinct prime divisor which are not contained in $\Bs(K_X+\sum_{i=1}^pa_i S_i)$ for all $0\leq a_i\leq 1$, 
\item[(iii)] the divisors $S_i$ span $\Div_{\mathbb R}(X)$ up to numerical equivalence, 
\item[(iv)] $\rho(X)\le \rho$ and $\rho(S_i)\le \rho$ for all $i=1,\dots,p$.
\end{enumerate}
Let $I$ be the total number of irreducible components of intersections of each two and each three of the divisors $S_1,\dots,S_p$. 

There exists a constant $C$ that depends only on $p$, $\rho$, $\varepsilon$ and $I$ such that for any $\Delta=\sum_{i=1}^p \delta_i S_i$ with $\delta_i\in [\varepsilon, 1-\varepsilon]$ and $(X,\Delta)$ terminal, the number of log terminal models of $(X,\Delta)$ is at most $C$. 
\end{theorem} 

The proof is an easy consequence of our main technical result, Theorem \ref{l_disconnected} below.

Now, say that two log smooth pairs $(X_1,\Delta_1)$ and $(X_2,\Delta_2)$ are of the same \emph{topological type} is there is a homeomorphism $\varphi\colon \map X_1.X_2.$ which is a homeomophism between $\Supp\Delta_1$ and $\Supp\Delta_2$. Then we have the following corollary.

\begin{corollary}\label{corollary}
Let $\varepsilon$ be a positive number. Let $\mathfrak X$ be the collection of all log smooth $3$-fold terminal pairs $(X,\Delta=\sum_{i=1}^p\delta_iS_i)$ such that $X$ is not uniruled, $\varepsilon\leq\delta_i\leq1-\varepsilon$ for all $i$, $S_1,\dots,S_p$ are distinct prime divisor not contained in $\Bs(K_X+\sum_{i=1}^pa_i S_i)$ for all $0\leq a_i\leq 1$, and $S_i$ span $\Div_{\mathbb R}(X)$ up to numerical equivalence. 

Then for every $(X_0,\Delta_0)\in\mathfrak{X}$ there exists a constant $N$ such that for every $(X,\Delta)\in\mathfrak{X}$ of the topological type as $(X_0,\Delta_0)$, the number of log terminal models of $(X,\Delta)$ is bounded by $N$.
\end{corollary}

These results lend strong support to a conjecture that the number of minimal models of a smooth projective threefold is completely determined only by its topology. This belief also has roots in other results in the field. According to philosophy starting with \cite{Kollar86}, vanishing and injectivity theorems in cohomology hold due to topological reasons, and Koll\'ar's effective basepoint freeness gives bounds that depend only on the dimension, see also the related Conjecture \ref{con1} below. The finite generation of adjoint rings can be proved only by using the Kawamata-Viehweg vanishing \cite{CL10a}, and the number of minimal models of a given pair is closely related to these rings \cite{CorLaz12,KKL12}.

In the proofs we use the full force of the $3$-dimensional MMP. Our main tools are Shokurov's log geography \cite{Shokurov96} and the techniques involved in the proof of termination of $3$-fold flips. The log geography has played an important role in studying the birational geometry of projective varieties: for instance, it was recently used to prove the Sarkisov Program for Mori fibre spaces \cite{HM10b}. We believe that a more accurate study of Fano threefolds combined with the results of this paper will give a new insight on the classification of Fano threefolds \cite{Cor09}. Furthermore, our results combined with the Cone conjecture suggest that there is a topological bound on the number of faces of the fundamental domain of the action of the group of birational automorphisms on the movable cone of a Calabi-Yau manifold; in particular, this relates to recent results on the Cone conjecture for Calabi-Yau log surface pairs \cite{Totaro10}.

\medskip 

Thus, our results are related to the following question:
\begin{question} \label{question}
Let $(X,\sum S_i)$ be a log smooth projective variety of dimension $n$, where $S_1,\dots,S_p$  are distinct prime divisors, and let $\varepsilon>0$ be a rational number.  

Does there exist a constant $M$ depending only on the topology of $X,S_1,\dots,S_p$ and on $\varepsilon$ such that the number of log terminal models of $(X,B)$, where the coefficients of $B$ lie in the interval $[\varepsilon,1-\varepsilon]$, is bounded by $M$?
\end{question}

Note that this number is expected to be finite, cf.\ Theorem \ref{t_shokurov}. We give a positive answer to Question \ref{question} in the case of  surfaces in Theorem \ref{t_su} below, and this provides an effective version of the finite generation of adjoint rings on surfaces, cf.\ Theorem \ref{t_efgs}, which generalises some of the results from \cite{CZ12}. Furthermore, our main result, Theorem \ref{l_disconnected}, gives a positive answer to Question \ref{question} in the case of non-uniruled terminal threefold pairs $(X,B)$ such that the prime divisors contained in the support of $B$ span $\Div_{\mathbb R}(X)$ up to numerical equivalence.
\section{Preliminary results}

We work over the field of complex numbers $\mathbb C$. The size of a set $S$ is denoted by $\#S$. We denote by $\mathbb R_+$ and $\mathbb Q_+$ the sets of non-negative real and rational numbers. The notation $N=N(a_1,\dots,a_k)$ means that the constant $N$ depends only on the parameters $a_1,\dots,a_k$.

\subsection{Divisors, valuations and models}
We first recall some standard definitions \cite{KM98,Lazarsfeld04,BCHM10}.

Let $X$ be a normal projective variety and $\mathbf R\in \{\mathbb Z,\mathbb Q,\mathbb R\}$. We denote by $\Div_{\mathbf R}(X)$ the group of $\mathbf R$-divisors on $X$, and $\sim_{\mathbf R}$ denotes $\mathbf R$-linear equivalence of $\mathbb R$-divisors. If $A=\sum a_iC_i$ is an $\mathbb{R}$-divisor on $X$, then $\rfdown A.=\sum \rfdown a_i.C_i$ is the round-down of $A$. 

If $X$ is a normal projective variety and if $D$ is an integral divisor on $X$, we denote by $\bs|D|$ the base locus of $D$. If $S$ is a prime divisor on $X$ such that $S\nsubseteq\bs|D|$, then $|D|_S$ denotes the image of the linear system $|D|$ under restriction to $S$. If $D$ is an $\mathbb R$-divisor on $X$, we denote
$$\Bs(D)=\bigcap_{D\sim_\mathbb R D'\geq0}\Supp D',$$
and we call $\Bs(D)$ the {\em stable base locus} of $D$. If $A$ is any ample divisor on $X$, then $\Bs_+(D)=\bigcap_{\varepsilon>0}\Bs(D-\varepsilon A)$ is the \emph{augmented base locus} of $D$, and we clearly have $\Bs(D)\subseteq\Bs_+(D)$.

A {\em log pair} $(X, \Delta)$ consists of a normal variety $X$ and an $\mathbb R$-divisor $\Delta\ge 0$ such that $K_X+\Delta$ is $\mathbb R$-Cartier. We say that $(X,\Delta)$ is {\em log smooth} if $X$ is smooth and $\Supp\Delta$ has simple normal crossings. A projective birational morphism $f\colon\map Y.X.$ is a {\em log resolution} of the pair $(X, \Delta)$ if $Y$ is smooth, $\Exc f$ is a divisor and the support of $f_*^{-1}\Delta+\Exc f$ has simple normal crossings. A birational map $f\colon \rmap X.Y.$ between normal projective varieties is a \emph{contraction} if $f$ does not extract a divisor.

Let $f\colon \rmap X.Y.$ be a birational contraction between normal projective varieties and let $D$ be an $\mathbb R$-Cartier divisor on $X$ such that $D_Y=f_*D$ is also $\mathbb R$-Cartier. Then $f$ is 
$D$-\emph{nonpositive} (respectively $D$-\emph{negative}) if for some resolution $(p,q)\colon \map W.X\times Y.$ of $f$, we may write
$$p^*D=q^*D_Y+E,$$
where $E\ge 0$ is $q$-exceptional (respectively $E\ge 0$ is $q$-exceptional and $\Supp E$ contains the strict transforms of all $f$-exceptional divisors).     

A {\em geometric valuation\/} $\Gamma$ on a normal variety $X$ is a valuation on $k(X)$ given by the order of vanishing at the generic point of a prime divisor on some birational model $f\colon Y\to X$. If $D$ is an $\mathbb R$-Cartier divisor on $X$, by abusing notation we use $\mult_\Gamma D$ to denote $\mult_\Gamma f^*D$. In this paper, we do not distinguish between a prime divisor on a birational model of $X$ and the corresponding geometric valuation. In this notation, if $\Gamma$ is a prime divisor on $Y$, then $c_X(\Gamma)=f(\Gamma)$ is the \emph{centre} of $\Gamma$ on $X$.

Given a log pair $(X,\Delta)$ and a geometric valuation $\Gamma$, let $f\colon\map Y.X.$ be the birational morphism such that $\Gamma$ is a divisor on $Y$. The {\em discrepancy} $a(\Gamma,X,\Delta)$ of $(X,\Delta)$ with respect to $\Gamma$ is the coefficient of $\Gamma$ in the divisor $K_Y-f^*(K_X+\Delta)$. Discrepancies are used to define singularities of pairs that we use in this paper (canonical, terminal, klt, log canonical, plt), see \cite{KM98}.

\medskip 

We will use the following lemma  in Section \ref{sec:threefolds}.

\begin{lemma}\label{lem:discrepbound}
Let $(X,\sum_{i=1}^p b_iS_i)$ be a log smooth terminal threefold pair, where $S_1,\dots,S_p$ are distinct 
prime divisors. Let
$$f\colon \rmap X.X'.$$
be a birational contraction to a  terminal threefold $X'$. Let $S_i'$ be the proper transform of $S_i$ in $X'$ for 
every $i$. Let $Y$ be a smooth variety, let $g\colon Y\longrightarrow X$ be a birational morphism, and let $E\subseteq Y$ be an $(f\circ 
g)$-exceptional prime divisor such that the centre of $E$ on $X'$ is a curve. Then
\begin{equation}\label{e_a}
a\Big(E,X',\sum_{i=1}^pb_iS_i'\Big)=a(E,X',0)-\sum_{i=1}^pb_i\mult_E S_i',
\end{equation}
where $0<a(E,X',0)<\rho(Y/X')$.
\end{lemma}
\begin{proof}
We easily calculate \eqref{e_a}
and note that $a(E,X',0)$ is a positive integer since $X'$ is terminal. Let $T\subseteq X'$ be a general ample surface, and let $W$ be its proper transform on $Y$. After possibly replacing $X$ with a smaller open subset of $X$, we may assume that $T\cap c_{X'}(E)$ is a smooth point of $X'$. 
Then the induced map $W\longrightarrow T$ is a birational morphism and $W$ is obtained from $T$ by blowing up $\rho(W/T)$ times. 

Let $(p,q)\colon Z\longrightarrow Y\times X'$ be a resolution of $f\circ g$. Then since $T$ is general we have $T':=q^*T=q^{-1}_*T$, and hence
$$K_Z+T'=q^*(K_{X'}+T)+\Gamma$$
for some $q$-exceptional  divisor $\Gamma\geq0$. Restricting this equality to $T'$ and pushing forward to $X$, we obtain $a(E,X',0)=a(W\cap E,T,0)$. Since $T\cap c_{X'}(E)$ is smooth, it is easy to see from the discrepancy formulas that $a(W\cap E,T,0)\leq\rho(W/T)$. Finally, observe that since $T$ is general, $\rho(W/T)$ is bounded by the number of $(f\circ g)$-exceptional divisors on $Y$, hence it is bounded by $\rho(Y/X')$. 
\end{proof}

\begin{lemma}\label{l_canonical}
Let $(X,\Delta)$ be a canonical projective pair, and let $f\colon \rmap X.Y.$ be a $(K_X+\Delta)$-nonpositive birational contraction. Assume that $f$ does not contract any component of $\Delta$, and let $\Delta_Y=f_*\Delta$. 

Then $(Y,\Delta_Y)$ is canonical. Additionally, if $f$ is $(K_X+\Delta)$-negative and $(X,\Delta)$ is terminal, then $(Y,\Delta_Y)$ is terminal. 
\end{lemma}
\begin{proof}
This follows easily from the definitions.
\end{proof}

The following result is inspired by \cite[Proposition 2.36]{KM98} and \cite[Lemma 1.5]{AHK06}.

\begin{lemma}\label{l_echo}
Let $(X,\Delta=\sum_{i=1}^p a_i S_i)$ be a  $3$-dimensional log smooth terminal pair with $0<a_i<1$, and let $Z\subseteq \sum_{i=1}^p S_i$ be a union of $m$ curves. Let $I$ be the total number of points of intersection of each three of the divisors $S_1,\dots,S_p$.

Then there exists a constant $N=N(m,p,a_1,\dots,a_p,I)$ such that the number of geometric valuations $E$ on $X$ with $c_X(E)\subseteq Z$ and $a(E,X,\Delta)<1$ is bounded by $N$. Furthermore, the number of smooth blow-ups needed to realise the valuations is bounded by $N$.
\end{lemma}
\begin{proof}
After possibly replacing $X$ by a smaller open subset, we may assume that $S_i\cap S_j\subseteq Z$ for any distinct $i,j\in\{1,\dots,p\}$. Since $(X,\Delta)$ is log smooth, by first blowing up intersections of triples of components $S_i$, and then intersections of each two of them, we obtain a composition of $M=M(m,p,I)$ blowups $f\colon Y\longrightarrow X$ such that we may write
$$K_Y+\Gamma=f^*(K_X+\Delta)+E_Y,$$
where $\Gamma$ and $E_Y$ are effective $\mathbb R$-divisors with no common components, $(Y,\Gamma)$ is log smooth, $E_Y$ is $f$-exceptional and the components of $\Gamma$ are pairwise disjoint. In particular, there are at most $C$ prime divisors $E$ on $Y$ such that $a(E,X,\Delta)<1$. Also, note that by discrepancy formulas, the discrepancies $a(E,X,\Delta)$ which lie in the interval $(0,1)$ are of the form $2-a_i-a_j-a_k$ or $1-a_i-a_j$ for some pairwise different $i,j,k$. 

It remains to count valuations which are exceptional over $Y$. Let $g\colon W\longrightarrow X$ be a log 
resolution which dominates $Y$, and let $W'\longrightarrow W$ be a blowup along a smooth centre with 
exceptional divisor $F$. Then it is easy to see that if $a(F,X,\Delta)<1$, then $c_W(F)$ is the intersection 
of the proper transform of some $S_i$ and some prime divisor $G$ on $Y$ with $0<a(G,X,\Delta)<1$. 

For each curve $C\subseteq Z$, if $f^{-1}$ is an isomorphism at the generic point of $C$, let $C'\subseteq Y$ be the unique curve isomorphic to $C$ at the generic point of $C'$; otherwise, let $C'$ be the union of curves on $Y$ which map onto $C$, and which are of the form $f^{-1}_*S_i\cap F$ for some prime divisor $F\subseteq Y$ with $0<a(F,X,\Delta)<1$. Hence, there are at most $m+mM$ such curves, let $Z'$ be their union, and by shrinking $X$ we may assume that all the curves in $Z'$ are smooth. Then, similarly as in \cite[Example 1.4]{AHK06}, there are at most $N'=N'(m,M,a_1,\dots,a_p)$ valuations over $Y$ with discrepancy smaller than $1$ and whose centres lie in $Z'$. Now set $N=N'+m$.
\end{proof}

Let $(X,\Delta)$ be a klt pair. A birational contraction $f\colon \rmap X.Y.$ is a \emph{log terminal model} for  $(X,\Delta)$ if $f$ is $(K_X+\Delta)$-negative, $Y$ is $\mathbb Q$-factorial and $K_Y+f_*\Delta$ is nef. If $K_Y+f_*\Delta$ is semiample, then it  defines a fibration $g\colon \map Y.Z.$, and the induced map $\rmap X.Z.$ is the {\em ample model} of $(X,\Delta)$. A log terminal model of $(X,0)$ is called a \emph{minimal model} of $X$.

The following result is well known, see for instance \cite[Lemma 2.13]{CZ12}:

\begin{lemma}\label{l_fact}
Let $(X,\Delta)$ be a log smooth log canonical surface pair such that $K_X+\Delta$ is pseudoeffective, and let $f\colon\map X.Y.$ be the log terminal model of $(X,\Delta)$. 
 
Then there exist birational morphisms $g\colon \map X.Z.$ and $h\colon \map Z.Y.$ such that $f=h\circ g$ and such that 
\begin{enumerate}
\item[(i)] $g$ is a composite of contractions of $(-1)$-curves, and
\item[(ii)] $h$ contracts only curves contained in the support of $g_*\Delta$.
\end{enumerate}
\end{lemma}

Throughout the paper, under the \emph{Minimal Model Program in dimension $n$} we assume that each log canonical pair $(X,\Delta)$ of dimension $n$ with $K_X+\Delta$ pseudoeffective admits a log terminal model and the ample model. 

Let $(X,\Delta)$ be a klt pair of dimension $n$, and let $f\colon\rmap X.Y.$ be a log terminal model of $(X,\Delta)$. If the Minimal Model Program holds in dimension $n$, then the prime divisors contracted by $f$ are precisely those that are contained in $\Bs(K_X+\Delta)$. The following lemma, which will be extensively used in Section \ref{sec:threefolds}, establishes a similar link between ample models and the augmented base loci.

\begin{lemma}\label{l_bplus} 
Let $X$ be a smooth projective threefold and let $D$ be a big $\mathbb Q$-divisor on $X$.  Let $f\colon \rmap X.Y.$ be the ample model of $D$.

Then $\Bs_+(D)$ coincides with the exceptional locus of $f$. 
\end{lemma}
\begin{proof}
The result follows immediately from \cite[Theorem A]{BCL13}.
\end{proof}

In special circumstances, the restriction of an MMP for a pair $(X,\Delta)$ to a prime divisor $S$ on $X$ induces an MMP on $S$. The following lemma is just a minor reformulation of \cite[Lemma 4.1]{BCHM10}, and follows from the proof of that result.

\begin{lemma}\label{lem:BCHM}
Let $(X,S+B)$ be a log smooth pair, where $S$ is a prime divisor and $\lfloor B\rfloor=0$, and let $\varphi\colon\rmap X.X'.$ be a weak log canonical model of $K_X+S+B$. Assume that $\varphi$ does not contract $S$, let $S'=\varphi_*S$ and $B'=\varphi_*B$, and let $\sigma\colon\rmap S.S'.$ be the induced birational map. Define a divisor $\Psi$ on $S'$ by $(K_{X'}+S'+B')|_{S'} = K_{S'} +\Psi$. 

If $(S,B|_S)$ is terminal, then there is a divisor $\Xi\leq B|_S$ such that $\sigma_*\Xi=\Psi$ and $\sigma$ is a weak log canonical model of $K_S+\Xi$.
\end{lemma}

The next lemma, combined with Lemma \ref{lem:BCHM}, shows that under certain assumptions, the restriction of the ample model is again the ample model on the restriction.

\begin{lemma}\label{lem:canonical}
Let $(X,S+B)$ be a plt pair, where $S$ is a prime divisor and $\lfloor B\rfloor=0$. Assume that $D=K_X+S+B$ is semiample, and let $f\colon\map X.Y.$ be the corresponding fibration. Assume that $f(S)\neq Y$ and let $g=f|_S$. 

Then $g$ is the semiample fibration associated to $D_{|S}$.
\end{lemma}
\begin{proof}
Fix a sufficiently divisible positive integer $m$ such that $f$ is the map associated to the linear system $|mD|$, and let $A$ be an ample $\mathbb Q$-divisor on $Y$ such that $D=f^*A$. Then $g$ is the map associated to the linear system $|mD|_S$, and it is enough to show that $|mD|_S=|mD_{|S}|$. From a long exact sequence in cohomology, this in turn is equivalent to showing that the map
$$H^1(X,mD-S)\longrightarrow H^1(X,mD)$$
is injective. Since $mD-S=K_X+B+(m-1)f^*A$, this follows from \cite[(10.19.3)]{Kollar93}.
\end{proof}

\begin{remark}
The assumption $f(S)\neq Y$ in Lemma \ref{lem:canonical} is necessary. Indeed, let $Y$ be a curve of genus $\geq2$. Let $\mathcal E$ be a sufficiently ample vector bundle of rank $2$ on $Y$, set $X=\p(\mathcal E)$, and let $f\colon \map X.Y.$ be the projection map. Then, by assumption, the line bundle $\xi=c_1(\mathcal O(1))$ is very ample, and let $S\in|2\xi|$ be a general section. If $G=c_1(\mathcal E)$, then $K_X+S=f^*(K_Y+G)$, and since $K_Y+G$ is ample, $f$ is the semiample fibration associated to $K_X+S$. However, the general fibre of $f$ meets $S$ in two points, thus $f|_S$ does not have connected fibres. 
\end{remark}

\subsection{Convex geometry}

Let $\mathcal{C}\subseteq \mathbb R^p$ be a convex set.
A subset $F\subseteq\mathcal C$ is a \emph{face} of $\mathcal{C}$ if $F$ is convex, and whenever $tu+(1-t)v\in F$ for some $u,v\in\mathcal C$ and $0<t<1$, then $u,v\in F$.  Note that $\mathcal C$ is itself a face of $\mathcal C$.
We say that $x\in \mathcal{C}$ is an \emph{extreme point} of $\mathcal C$ if $\{x\}$ is a face of $\mathcal{C}$.

A \emph{polytope} in $\mathbb R^p$ is a compact set which is the intersection of finitely many half spaces. A polytope is \emph{rational} if it is an intersection of finitely many rational half spaces. A \emph{rational polyhedral cone} in $\mathbb R^p$ is a convex cone spanned by finitely many rational vectors.

\begin{lemma}\label{l_comb}
Let $\mathcal C\subseteq \mathbb R^p$ be a rational polytope  which is defined by half-spaces 
$$\big\{(x_1,\dots,x_p)\in \mathbb R^p\mid \sum\nolimits_{j=1}^p\alpha_{ij}x_j\ge \beta_i\big\}$$ 
for $i=1,\dots,\ell$, where $\alpha_{ij}$ and $\beta_i$ are integers. Let $M$ be a positive integer such that 
$$\alpha_{ij}\ge -M \qquad\text{and}\qquad |\beta_i|<M$$ 
for all $i,j$. Pick  a positive real number $\varepsilon<1$.

Then there exists a positive integer $m$ which depends only on $M$, $p$ and $\varepsilon$ (but not on $\mathcal C$), such that for every extreme point $v$ of $\mathcal C$ which is contained in $[\varepsilon,1]^p$, the point $mv$ is integral.
\end{lemma}

\begin{proof}
Since $v=(v_1,\dots,v_p)$ is an extreme point of $\mathcal C$, after relabelling we may assume that $\sum_{j=1}^p\alpha_{ij} v_j=\beta_i$ for $i=1,\dots,p$. Denoting by $A$ the ($p\times p$)-matrix $(\alpha_{ij})$, we may additionally assume that the rows of $A$ are linearly independent over $\mathbb R$.  In particular, $\det A\neq0$ and Cramer's rule implies that $\det A\cdot v$ is integral. By assumption, we have
$$\sum_{\alpha_{ij}<0}\alpha_{ij}+\varepsilon\sum_{\alpha_{ij}>0} \alpha_{ij} \le  \sum_{j=1}^p\alpha_{ij}v_j = \beta_i<M,$$
and since $\alpha_{ij}\ge -M$, we have
$$|\alpha_{ij}|<\frac{Mp}\varepsilon\qquad \text{for all}\quad i,j=1,\dots,p.$$ 
Therefore, $\det A$ is bounded by a constant $m_0$ which depends on $M$, $p$ and $\varepsilon$, and the claim follows by taking $m=m_0!$. 
\end{proof}

\begin{definition}
Let $\mathcal P_1,\mathcal P_2\subseteq\mathbb R^p$ be polytopes of dimension $p$. We say that $\mathcal P_i$ are \emph{adjacent} if $\mathcal P_1\cap\mathcal P_2$ is a codimension one face of both $\mathcal P_1$ and $\mathcal P_2$.

Let $\mathcal P=\bigcup_{i=1}^k\mathcal P_i$ be a (not necessarily convex) finite union of polytopes. We say that $\mathcal P_i$ and $\mathcal P_j$ are \emph{adjacent-connected} if there exist indices $i_1,\dots,i_q$ such that $i_1=i$, $i_q=j$, and $\mathcal P_{i_s}$ and $\mathcal P_{i_{s+1}}$ are adjacent for every $s=1,\dots,q-1$. The equivalence classes of this relation are called \emph{adjacent-connected components}. If the whole $\mathcal P$ belongs to one such component, we say that $\mathcal P$ is also adjacent-connected. A \emph{face} of $\mathcal P$ is a face of any $\mathcal P_i$ which is not contained in the interior of $\mathcal P$.
\end{definition}

\begin{lemma}\label{lemma:polytopes}
Let $\mathcal Q\subseteq[0,1]^p\subseteq \mathbb R^p$ be a polytope containing the origin, and let $\mathcal C_1,\dots,\mathcal C_\ell$ be $p$-dimensional polytopes with pairwise disjoint interiors such that $\mathcal Q=\bigcup_{i=1}^\ell\mathcal C_i$. Let $\mathcal P_1,\dots,\mathcal P_k\subseteq \mathcal Q$ be $p$-dimensional polytopes such that
\begin{equation}\label{eq:polytopes}
(\mathcal P_i+\mathbb R_+^p)\cap \mathcal Q\subseteq \mathcal P_i
\end{equation}
for all $i$.  For any subset $I\subseteq \{1,\dots,k\}$, denote by $\mathcal R_I$ the closure of $\bigcup_{i\in I} \mathcal P_i \backslash \bigcup_{j\not\in I}\mathcal P_j$, and let $\mathcal R_0$ denote the closure of $\mathcal Q\setminus \bigcup_{i=1}^k \mathcal P_i$. Assume that each adjacent-connected component of every $\mathcal R_I$ and of $\mathcal R_0$ with respect to the covering $\mathcal Q=\bigcup_{i=1}^\ell\mathcal C_i$ is the union of at most $m$ polytopes $\mathcal C_i$. 

Then there exists a constant $M=M(k,m)$ such that $\ell\le M$. 
\end{lemma}
\begin{proof}
If $x=(x_1,\dots,x_p)\in \mathcal R_0$ and $y=(y_1,\dots,y_p)\in \mathcal Q$ are such that $y_i\le x_i$ for all $i=1,\dots,p$, then $y\in \mathcal R_0$
 by \eqref{eq:polytopes}. Therefore, the set $\mathcal R_0$ is adjacent-connected, and hence it contains at most $m$ polytopes $\mathcal C_i$.

For any $d=1,\dots,p$, let $\mathcal J_d$ be the set of codimension $d$ faces of $\mathcal R_0$ which are not contained in the boundary of $\mathcal Q$. Since the polytopes $\mathcal C_i$ and $\mathcal P_j$ are convex, and $\mathcal R_0$ contains at most $m$ polytopes $\mathcal C_i$, it follows that each $\mathcal P_j$ contains at most $m$ elements of $\mathcal J_{1}$, and hence $\#\mathcal J_1\leq mk$. Now, if $d>1$, each element of $\mathcal J_{d-1}$ contains at most $\#\mathcal J_{d-1}$ elements of $\mathcal J_d$, and therefore $\#\mathcal J_d\leq (\#\mathcal J_{d-1})^2$. This shows that $\#\mathcal J_d\le (mk)^{2^{d-1}}$.

Since
$$\textstyle\bigcup_{i\in I} \mathcal P_i \big\backslash \bigcup_{j\not\in I}\mathcal P_j=\bigcup_{i\in I} \big(\mathcal P_i \big\backslash \bigcup_{j\not\in I}\mathcal P_j\big),$$
it is enough to bound the number of adjacent-connected components of each set $\mathcal P_i \backslash \bigcup_{j\not\in I}\mathcal P_j$. The statement is trivial for $k=1$, hence by induction we may assume that $I=\{1,\dots,k\}$ and, without loss of generality, that $i=1$. For any element $F\in\mathcal J_1$, set $F_1=F\cap \mathcal P_1$ and by  \eqref{eq:polytopes}  we have that $\mathcal F_1:=(F_1+\mathbb R_+^p)\cap \mathcal Q\subseteq \mathcal P_1$. Thus, it is easy to see that
$$\textstyle\mathcal P_1 \big\backslash \bigcup_{j=2}^k\mathcal P_j=\bigcup_{F\in\mathcal J_1}\big(\mathcal F_1\big\backslash\bigcup_{j=2}^k\mathcal P_j\big),$$
hence it is enough to bound the number of adjacent-connected components contained in $\mathcal F_1\backslash\bigcup_{j=2}^k\mathcal P_j$. Again by \eqref{eq:polytopes}, it is enough to bound the number of adjacent-connected components of $F_1\backslash\bigcup_{j=2}^k\mathcal P_j$, with respect to the induced topology on $F_1$. Note that every codimension $d-1$ face of an adjacent-connected component of $F_1\backslash\bigcup_{j=2}^k\mathcal P_j$ is an element of $\mathcal J_d$. Hence, the number of such adjacent-connected components is bounded by a constant which depends only on all $\#\mathcal J_d$, and the lemma follows. 
\end{proof}

\subsection{Shokurov's geography of models and adjoint rings}

\begin{definition}\label{d_lv}
Let $(X,\sum_{i=1}^pS_i)$ be a log smooth projective pair, where $S_1,\dots,S_p$ are distinct prime divisors, and let $V=\sum_{i=1}^p \mathbb RS_i\subseteq \Div_{\mathbb R}(X)$. Let $0<\varepsilon<1/2$. We denote
$$\textstyle\mathcal L(V)=\big\{\sum_{i=1}^p a_i S_i\in V\mid a_i\in [0,1]\big\},\quad \mathcal E(V)=\{\Delta\in \mathcal L(V)\mid K_X+\Delta\sim_{\mathbb R} D\ge0\},$$
and 
\begin{align*}
\mathcal L_\varepsilon(V)&=\big\{\sum_{i=1}^p a_iS_i\in V\mid a_i\in [\varepsilon,1-\varepsilon]\big\},\\ 
\mathcal L_\varepsilon^{\can}(V)&=\{\Delta\in \mathcal L_\varepsilon(V)\mid (X,\Delta)\text{ is canonical}\}.
\end{align*}
The sets $\mathcal L(V)$, $\mathcal L_\varepsilon(V)$ and  $\mathcal L_\varepsilon^{\can}(V)$ are clearly rational polytopes. Note that since $(X,\sum_{i=1}^pS_i)$ is log smooth, if the dimension of $\mathcal L_{\varepsilon}^{\can}$ is $p$ and $\Delta$ is contained in its interior, then $(X,\Delta)$ is terminal. Finally, given a birational contraction $f\colon \rmap X.Y.$, let $\mathcal C_{f}(V)$ denote the closure in $\mathcal L(V)$ in the standard topology of the set
$$\{\Delta\in \mathcal E(V) \mid f \text { is a log terminal model of } (X,\Delta)\}.$$
\end{definition}

The following is \cite[Theorem 3.4]{CS11}.

\begin{theorem}\label{t_shokurov}
Assume the MMP in dimension $n$. Let $(X,\sum_{i=1}^pS_i)$ be a log smooth projective pair, where $S_1,\dots,S_p$ are distinct prime divisors, and let $V=\sum_{i=1}^p \mathbb RS_i\subseteq \Div_{\mathbb R}(X)$.

Then there exist birational contractions $f_i\colon\rmap  X.Y_i.$ for $i=1,\dots,k$, such that $\mathcal C_{f_1}(V),\dots,\mathcal C_{f_k}(V)$ are rational polytopes and 
$$\mathcal E(V)=\bigcup_{i=1}^k \mathcal C_{f_i}(V).$$ 
In particular, $\mathcal E(V)$ is a rational polytope. 
\end{theorem}

Together with the relative version of the Cone conjecture \cite{Kawamata97a}, the relative version of the previous theorem implies finiteness of minimal models up to isomorphism. The following theorem is folklore, but we include the proof for the benefit of the reader. The proof below came out of discussions with C. Xu.

\begin{theorem}\label{t_finiteness}
Assume the MMP in dimension $n$ and the relative Cone conjecture in dimension $n$. Let $X$ be a terminal projective variety of dimension $n$. 

Then the number of minimal models of $X$ is finite up to isomorphism.
\end{theorem}
\begin{proof}
Replacing $X$ by a minimal model, we may assume that $K_X$ is semiample, and let $X\longrightarrow S$ be the canonical model. If $Y$ is another minimal model of $X$ and $A\subseteq Y$ is a very ample divisor over $S$, then the map $\varphi\colon X\dashrightarrow Y$ is an isomorphism in codimension $1$, the divisor $D=\varphi^*A\subseteq X$ is movable over $S$ and $Y\simeq\Proj_S R(X/S,D)$. Let $\Pi$ be a fundamental domain for the action of $\Bir(X/S)$ on the cone $\overline{\mov}(X/S)\cap\Eff(X/S)$. Then there exists $g\in\Bir(X/S)$ such that $g^*D\in\Pi$, and we have $R(X/S,D)\simeq R(X/S,g^*D)$ since $g$ is an isomorphism in codimension $1$. Replacing $D$ by $g^*D$, we may assume that $D\in\Pi$.

Let $D_1,\dots,D_r$ be effective divisors whose classes generate $\Pi$, let $S_1,\dots,S_k$ be all the prime divisors in the support of $\sum D_i$, let $V=\sum_{i=1}^p \mathbb RS_i\subseteq \Div_{\mathbb R}(X)$, and let $\Pi'\subseteq V$ be the inverse image of $\Pi$ under the natural map $V\longrightarrow N^1(X)_\mathbb R$. Note that $D$ belongs to set $\Pi'\cap\mathbb R_+\mathcal L(V)$ since the pair $(X,\varepsilon D)$ is klt for some $0<\varepsilon\ll1$. Since $K_X$ is trivial over $S$, by \cite[Theorem 3.4]{CS11} and \cite[Theorem 4.2]{KKL12}, there are finitely many cones $\mathcal C_i\subseteq V$ and contractions $f_i\colon X\dashrightarrow Z_i$ for $i=1,\dots,k$, such that $\Pi'\cap\mathbb R_+\mathcal L(V)=\bigcup\mathcal C_i$ and if $\Delta\in\mathcal C_i\cap \mathcal L(V)$, then $f_i$ is the ample model of $K_X+\Delta$ over $S$. In particular, there exists a cone $\mathcal C_i$ which contains $D$, and hence $Y\simeq Z_i$.
\end{proof}

Some of the polytopes $\mathcal C_{f_i}(V)$ are of special importance in this paper.

\begin{definition}\label{def:chambers}
Assume the MMP in dimension $n$. Let $(X,\sum_{i=1}^pS_i)$ be a log smooth projective pair of dimension $n$, where $S_1,\dots,S_p$ are distinct prime divisors and let $V=\sum_{i=1}^p \mathbb RS_i\subseteq \Div_{\mathbb R}(X)$. Let $f\colon \rmap X.Z.$ be a birational contraction. If $\mathcal C_i=\mathcal C_{f_i}(V)\cap\mathcal L_\varepsilon^{\can}(V)$ a polytope of dimension $p$ which intersects the interior of $\mathcal L_\varepsilon^{\can}(V)$, then $\mathcal C_i$ is called a \emph{terminal chamber} in $V$. 
\end{definition}

We recall the definition of adjoint rings from \cite{CL10a}.

\begin{definition}\label{d_adjointrings}
If $X$ is a smooth projective variety and $D_1,\dots,D_k$ are $\mathbb Q$-divisors on $X$, then we define 
$$\mathfrak R=R(X;D_1,\dots,D_k)=\bigoplus_{(m_1,\dots,m_k)\in \mathbb N^k}\textstyle H^0\big(X,\ring X.(\rfdown \sum m_iD_i.)\big).$$
The ring $\mathfrak R$ is \emph{generated in degree} $m$ if for every generator $\sigma$ of $\mathfrak R$ there exist $a_1,\dots,a_k\in \{0,\dots,m\}$ such that $\sigma\in H^0(X,\ring X.(\sum_{i=1}^k a_iD_i))$. 

Furthermore, if $\mathcal S\subseteq \Div_{\mathbb Q}(X)$ is a finitely generated monoid, then we define
$$R(X,\mathcal S)=\sum_{D\in \mathcal S}H^0(X,\ring X.(\rfdown D.)).$$
If $\mathcal P\subseteq \Div_{\mathbb R}(X)$ is a rational polyhedral cone, then $\mathcal S=\mathcal P\cap \Div(X)$ is a finitely generated monoid and we can define the adjoint ring associated to $\mathcal P$ as $R(X,\mathcal S)$. If divisors $D_1,\dots,D_k$ generate $\mathcal S$, there is the natural projection $\map R(X;D_1,\dots,D_k).R(X,\mathcal S).$, and we say that $R(X,\mathcal S)$ is generated in degree $m$ if $R(X;D_1,\dots,D_k)$ is. 
\end{definition}

The following conjecture seems to be folklore.

\begin{conjecture}\label{con1}
Let $(X,\Delta)$ be a projective klt pair of dimension $n$ with $K_X+\Delta$ nef, and let $k$ be a positive integer such that $k(K_X+\Delta)$ is Cartier. 

Then there exists a positive integer $m=m(n,k)$ such that the linear system $|m(K_X+\Delta)|$ is basepoint free.
\end{conjecture}

We spend a few words on this conjecture. It holds on surfaces by \cite[Lemma 2.6]{CZ12}. In general, the conjecture is clearly a post-Abundance problem. One of the main difficulties is when $\kappa(X,K_X+\Delta)=0$. If $K_X+\Delta$ is big, then it follows from Koll\'ar's effective basepoint free theorem \cite{Kollar93b} and if $K_X+\Delta$ is of intermediate Kodaira dimension, then the conjecture is related to \cite[Conjecture 7.13]{PS09}.

\begin{proposition}\label{p_ltog} 
Assume the MMP in dimension $n$ and Conjecture \ref{con1} in dimension $n$. Let $(X,\sum_{i=1}^pS_i)$ be a log smooth projective pair, where $S_1,\dots,S_p$ are distinct prime divisors, and let $V=\sum_{i=1}^p \mathbb RS_i\subseteq \Div_{\mathbb R}(X)$. Let $B_1,\dots,B_m\in\mathcal E(V)$, and denote by  $\mathcal C$ the rational polytope spanned by all $B_i$. Let $f_i\colon \rmap X.Z_i.$ be birational contractions as in Theorem \ref{t_shokurov}, and denote $\mathcal C_i=\mathcal C\cap \mathcal C_{f_i}(V)$. Let $\ell$ be a positive integer such that if $B$ is an extreme point of some $\mathcal C_j$, then $\ell B$ is integral. 

Then there is a constant $M=M(n,p,\ell)$ such that the ring 
$$\mathfrak R=R\big(X;M(K_X+B_1),\dots,M(K_X+B_m)\big)$$
is generated in degree $n+2$. 
\end{proposition}
\begin{proof}
Denote $D_i=\ell(K_X+B_i)$ and $\mathcal M=\sum_{i=1}^p\mathbb Z D_i$, and for each positive integer $N$ denote $\mathcal 
M^{(N)}=\sum_{i=1}^p\mathbb Z ND_i$. We first claim that there exist a constant $M'=M'(p,\ell)$ and generators
of $\mathcal M\cap\mathbb R_+(K_X+\mathcal C_i)$ for all $i$ which are of the form $M'(K_X+B)$ for some $B\in\mathcal L(V)$. 
Indeed, there is an obvious isomorphism $V\simeq \mathbb R^p$ which sends $\mathcal L(V)$ to the unit hypercube $[0,1]^p$. 
By assumption, each cone $\mathcal C_i$ is spanned by some $B_{ij}$, and each $B_{ij}$ corresponds to a point in $
(\frac1\ell\mathbb Z\cap[0,1])^p$. Hence, there are finitely many choices for these generators, and the claim follows. 

Therefore, by Conjecture \ref{con1} there exists a positive integer $M=M(n,M')$ such that for every $i$ and $j$, the divisor $M(f_i)_*(K_X+B_{ij})$ is basepoint free. In particular, each of the rings 
$$R(X,\mathcal C_i\cap\mathcal M^{(M/\ell)})\simeq R(Z_i,(f_i)_*(\mathcal C_i\cap\mathcal M^{(M/\ell)}))$$
is generated in degree $n+2$ by \cite[Proposition 2.11]{CZ12}. But then it is clear that the ring $\mathfrak R$ is generated in degree $n+2$.
\end{proof}
\section{Effective finite generation for adjoint rings on surfaces}

In this section we prove an effective version of the finite generation of adjoint rings on surfaces, cf.\ Definition \ref{d_adjointrings}, and we give a positive answer to Question \ref{question} in dimension $2$.

\begin{theorem}\label{t_efgs}
Let $(X,\sum_{i=1}^p S_i)$ be a log smooth projective surface pair, where $S_1,\dots,S_p$ are distinct prime divisors. 
Let $B_1,\dots,B_\ell$ be $\mathbb Q$-divisors on $X$ such that $\Supp B_i=\sum_{i=1}^p S_i$ and $\rfdown B_i.=0$ for every $i$, and let $k$ be a positive integer such that all $kB_i$ are Cartier.

Then there exists a positive integer $m=m(p,k)$ such that the ring 
$$R(X;m(K_X+B_1),\dots,m(K_X+B_\ell))$$
is generated in degree  $4$. 
\end{theorem}

\begin{proof}
Let $V=\sum_{i=1}^p\mathbb RS_i\subseteq \Div_{\mathbb R}(X)$. 
By Theorem \ref{t_shokurov}, for $i=1,\dots,q$ there exist contraction maps $f_i\colon \map X.Z_i.$ such that $\mathcal 
E(V)=\bigcup_{i=1}^q \mathcal C_{f_i}(V)$. 
Let $\mathcal C\subseteq \mathcal L(V)$ be the rational polytope spanned by $B_1,\dots,B_\ell$ and denote  
$$\mathcal C_{i}=\mathcal C_{f_i}(V)\cap \mathcal C.$$ 
Let $C_1,\dots,C_t$ be the extreme points of all such $\mathcal C_{i}$. Then by Proposition \ref{p_ltog} and by \cite[Lemma 
2.6]{CZ12}, it is enough to show that there exists a constant $d=d(p,k)$ such that all divisors $dC_i$ are Cartier. After possibly 
adding more divisors $B\in V$ such that $\rfdown B.=0$, $\Supp B=\sum_{i=1}^p S_i$ and $kB$ is Cartier,  we may assume 
that $B_1,\dots,B_\ell$ span $V$, and that all $\mathcal C_i$ are rational polytopes of dimension $p$. 

Fix an index $1\leq i_0\leq t$. If the point $C_{i_0}$ is an extreme point of $\mathcal C$, then $kC_{i_0}$ is Cartier by assumption. 
Otherwise, if $\mathcal F$ is a codimension one face of $\mathcal C_i$ for some $i$ containing $C_{i_0}$, then $\mathcal F$ is 
either a subset of a codimension one face of $\mathcal E(V)$, or $\mathcal F$ intersects the interior of $\mathcal C\cap \mathcal E(V)$. We claim that, for any such face $\mathcal F$,  there exist $1\leq k_0\leq q$ and a rational curve $\xi$ on $Z_{k_0}$ such that, setting $f=f_{k_0}$ and $Z=Z_{k_0}$, we have $f_*S_i\cdot \xi\neq 0$ for some $i$, and that for all $\Delta\in \mathcal F$, 
\begin{equation}\label{e_Z}
(K_{Z}+f_*\Delta)\cdot \xi=0\qquad \text{and}\qquad K_Z\cdot \xi\ge -2.
\end{equation} 

Assuming the claim, let us show how it implies the theorem. Denote  
$$\beta={-}K_{Z}\cdot \xi\qquad\text{and}\qquad \alpha_{j}=f_*S_j\cdot \xi \text{ for } j=1,\dots,p.$$
Let $B\in\mathcal C_{k_0}$ be a divisor such that $f$ is a log terminal model of $(X,B)$. By assumption, the coefficients of $kB_i$ are integers smaller than $k$ for all $i=1,\dots,\ell$, hence $\mult_{S_i}B\le 1-1/k$ for all $i$. If $S_i$ is $f$-exceptional for some $1\leq i\leq p$, then we have 
$$a(S_i,Z,0)\geq a(S_i,Z,f_*B)\ge a(S_i,X,B)= -1+\frac 1 k.$$
Thus, \cite[Proposition 2.14]{CZ12} together with Lemma \ref{l_fact} implies that there is a positive integer $r$ bounded by a constant which depends only on $k$ and $p$ such that all $r\alpha_{j}$ and $r\beta$ are integers. 

The hyperplane
$$\Pi=\{\textstyle \Delta=\sum_{j=1}^p d_jS_j\in V\mid \sum_{j=1}^p \alpha_{j} d_j=\beta\}\subseteq V$$
contains the face $\mathcal{F}$ by \eqref{e_Z}. By Lemma \ref{l_comb}, it is enough to bound all $\alpha_j$ from below and 
$|\beta|$ from above by a constant depending only on $k$.

Assume first that $\alpha_{j}\ge 0$ for all $j$. Then $f_*\Delta\cdot \xi>0$ for all $\Delta\in \mathcal F$ by the claim and since all $\Delta$ have the same support, and hence 
$0<\beta=-K_Z\cdot \xi \le 2$. 

Now assume that $\alpha_{j_0}<0$ for some $j_0$. Then $\xi=f_*S_{j_0}$ and $\xi^2<0$, and $\alpha_j\geq0$ for $j\neq j_0$. Fix $\Delta\in \mathcal F$, and denote $\gamma=\mult_\xi 
f_*\Delta$. Then $0<\gamma\le 1-1/k$ and
\begin{equation}\label{eq1}
{-}\beta+\gamma\xi^2=(K_{Z}+\gamma \xi)\cdot \xi\le (K_{Z}+f_*\Delta)\cdot \xi=0
\end{equation}
by \eqref{e_Z}, while
\begin{equation}\label{eq2}
{-}\beta+\xi^2=(K_{Z}+\xi)\cdot \xi \ge {-}2
\end{equation} 
by adjunction. In particular, \eqref{eq1} and \eqref{eq2} imply $(1-\gamma)\xi^2\ge -2$. Thus,
\begin{equation}\label{eq3}
\alpha_{j_0}=\xi^2\ge {-}\frac 2 {1-\gamma}\ge {-}2k\quad\text{and}\quad \beta\leq 2+\xi^2<2,
\end{equation}
and by \eqref{eq1} and \eqref{eq3} we have
$$\beta\ge \gamma \xi^2\ge  \xi^2\ge {-}2k,$$
which gives the desired bounds.

\medskip

It remains to show the claim stated above. If $\mathcal F$ is a face of $\mathcal E(V)$, then we set $k_0=i$. Note that $K_{Z_i}+(f_i)_*\mathcal F$ belongs to the boundary of the nef cone of $Z_i$, and hence, it defines a morphism with connected fibres which contracts a rational curve $\xi$ and such that $(K_{Z_i}+(f_i)_*\Delta')\cdot\xi>0$ for any $\Delta'$ in the interior of $\mathcal C_i$.  In particular $\xi$ satisfies \eqref{e_Z} and $(f_i)_*S_i \cdot \xi\neq 0$ for some $i=1,\dots,p$.

If $\mathcal F$ is also a face of $\mathcal C_{f_j}(W)$ for some $j\neq i$, then $K_{Z_i}+(f_i)_*\Delta$ and $K_{Z_j}+(f_{j})_*\Delta$ are nef for all $\Delta\in \mathcal F_W$. Let $Z_{ij}$ be the ample model of $(X,\Delta)$ for some $\Delta$ in the interior of $\mathcal F_W$.  Then by \cite[Theorem 3.3]{HM10b} or \cite[Theorem 4.2]{KKL12} there exist birational contractions $g_i\colon \map Z_i.Z_{ij}.$ and $g_j\colon\map Z_j.Z_{ij}.$ which are not both isomorphisms since $Z_i$ and $Z_j$ are not isomorphic. Without loss of generality, we may assume that $g_i$ is not an isomorphism and we set $k_0=i$. Since $K_{Z_i}+(f_i)_*\mathcal F$ belongs to the boundary of the nef cone of $Z_i$, the map $g_i$ contracts a rational curve $\xi$, and we conclude as above.
\end{proof}

Theorem \ref{t_efgs} immediately implies the following:
\begin{corollary}
Let $(X,\sum_{i=1}^p S_i)$ be a log smooth projective surface pair, where $S_1,\dots,S_p$ are distinct prime divisors, and let $V=\sum_{i=1}^p\mathbb RS_i\subseteq \Div_{\mathbb R}(X)$. Let $0<\varepsilon<1$ be a rational number, and let $B_1,\dots,B_\ell$ be the vertices of $\mathcal L_\varepsilon(V)$.

Then there exists a positive integer $m=m(p,\varepsilon)$ such that the ring 
$$R(X;m(K_X+B_1),\dots,m(K_X+B_\ell))$$ 
is generated in degree $4$. 
\end{corollary}

\begin{example}
The following example shows that the requirement in Theorem \ref{t_efgs} that the divisors $B_i$ have the same support is necessary. Let $m$ and $d$ be positive integers with $m$ even. We claim that there exists a smooth surface $X$ and a smooth prime divisor $S$ on $X$ such that the ring
$$\mathfrak R=R\Big(X;mK_X,m\Big(K_X+\frac 1 2S\Big)\Big)$$ 
is not generated in degree $d$. Indeed, let $Y$ be a smooth surface with $K_Y$ ample, and let $f\colon \map X.Y.$ be the blowup of $Y$ at a point with  exceptional divisor $E\subseteq X$. Let $S\subseteq X$ be a smooth divisor such that $s=S\cdot E> 2d+2$. It is easy to see that $E\nsubseteq\Bs(K_X+tS)$ for any $\frac1s\leq t\leq\frac12$, and $E\subseteq\Bs(K_X+tS)$ for $0\leq t<\frac1s$. Then there is a generator of $\mathfrak R$ in the vector space $H^0(X,q(K_X+\frac1s S))$ for some sufficiently divisible positive integer $q$: otherwise, every element of $H^0(X,q(K_X+\frac1s S))$ would vanish at $E$. Now, if $\ell$ and $p$ are positive integers such that 
$$q\Big(K_X+\frac1s S\Big)=\ell mK_X+pm\Big(K_X+\frac12 S\Big),$$
by intersecting with $E$ we get $\ell=p(\frac{s}2-1)>d$, which proves the claim.
\end{example}

\begin{example} 
The following example shows that the requirement in Theorem \ref{t_efgs} that $\rfdown B_i.=0$ is also necessary. Let $m$ and $d$ be positive integers with $m$ even, and let $s\geq 4d+4$ be a positive integer. Let $X$ be the Hirzebruch surface $\mathbb F_s$, let $E$ be the smooth rational curve in $X$ with $E^2=-s$ and let $f\colon X\longrightarrow Y$ be the morphism which contracts $E$. Then $f^*K_Y=K_X+(1-\frac2s)E$. Let $A$ be a smooth ample divisor on $Y$ not passing through $f(E)$ such that $K_Y+\frac12 A$ is ample and $K_X+\frac12f^*A+tE$ is nef for $\frac12\leq t\leq 1-\frac2s$. The pair $(X,f^*A+E)$ is log smooth, and consider the ring 
$$\mathfrak R=R\Big(X;m\Big(K_X+\frac12 f^*A + \frac12 E\Big),m\Big(K_X+\frac12 f^*A+E\Big)\Big).$$
Then $E\nsubseteq\Bs(K_X+\frac12f^*A+tE)$ for $\frac12\leq t\leq 1-\frac2s$, and $E\subseteq\Bs(K_X+\frac12f^*A+tE)$ for $1-\frac2s<t\leq 1$. We conclude that $\mathfrak R$ is not generated in degree $d$ as in the previous example. 
\end{example}

\medskip 

By using the same methods as in the proof of Theorem \ref{t_efgs}, we can easily get the following:

\begin{theorem}\label{t_su}
Let $p$  be a positive integer and let $0<\varepsilon<1$ be a rational number. 

Then there exists a constant $N=N(p,\varepsilon)$ such that if $(X,\sum_{i=1}^p S_i)$ is a log smooth projective surface pair, then the number of log terminal models of $(X,B)$ with $B\in \mathcal E(V)\cap \mathcal L_\varepsilon (V)$, where $V=\sum_{i=1}^p\mathbb RS_i\subseteq \Div_{\mathbb R}(X)$, is bounded by $N$. 
\end{theorem}
\begin{proof}
Let $B_1,\dots,B_{2^p}$ be the vertices of $\mathcal L_{\varepsilon/2}(V)$. If $k$ is the denominator of $\varepsilon/2$, then $kB_i$ is Cartier for all $i$. By Theorem \ref{t_shokurov}, for $i=1,\dots,q$ there exist contraction maps $f_i\colon \map X.Z_i.$ such that $\mathcal E(V)=\bigcup_{i=1}^q \mathcal C_{f_i}(V)$, and denote  
$$\mathcal C_{i}=\mathcal C_{f_i}(V)\cap \mathcal C.$$ 
Let $C_1,\dots,C_t$ be the extreme points of all such $\mathcal C_{i}$. Then by the proof of Theorem \ref{t_efgs}, there exists a constant $d=d(p,k)$ such that all divisors $dC_i$ are Cartier. There is an obvious isomorphism $V\simeq \mathbb R^p$ which sends $\mathcal L(V)$ to the unit hypercube $[0,1]^p$, and each $C_i$ corresponds to a point in $
(\frac1d\mathbb Z\cap[0,1])^p$. Hence, there are finitely many choices for these generators, and the theorem follows. 
\end{proof}

If we consider non-rational surfaces, the result can be strengthened as in Theorem \ref{t_surface} below. First we give a topological bound on the number of $(-1)$-curves on such surfaces.

\begin{lemma}\label{l_hg}
Let $X$ be a smooth projective surface which is not rational. 

Then there exists a constant $A=A(c_1(X)^2,c_2(X))$ such that the number of $(-1)$-curves on $X$ is bounded by $A$.  
\end{lemma}

\begin{proof}
We first assume that $X$ is not uniruled, and let  $f\colon \map X.Y.$ be the minimal model of $X$. Then, since $X$ and $Y$ are smooth, we have $K_X=f^*K_Y+E$, where $E\geq0$ contains the whole exceptional locus. If $C$ is a $(-1)$-curve on $X$, then $E\cdot C<0$ since $K_Y$ is nef, hence $C$ is a component of $E$. In particular, the number of $(-1)$-curves on $X$ is bounded by $\rho(X/Y)=c_1(Y)^2-c_1(X)^2$. From Noether's formula and since $Y$ is not uniruled, we have $c_1(Y)^2-c_1(X)^2=c_2(X)-c_2(Y)\leq c_2(X)$. Set $A_1=c_2(X)$.  

Assume now that $X$ is uniruled. Then, since $X$ is not rational, there exists a birational morphism $f\colon \map X.Y.$ to a relatively minimal ruled surface $Y\longrightarrow C$ over a projective curve of genus $g>0$. Note that 
$$g=h^1(X,\ring X.)=1-\frac {c_1(X)^2+c_2(X)}{12}$$ 
by Noether's formula, and hence
$$c_1(Y)^2=8(1-g)=\frac {2\big(c_1(X)^2+c_2(X)\big)}{3}.$$ 
By the uniqueness of minimal models on surfaces, the map $f$ factors through any Castelnuovo contractions of a $(-1)$-curve on $X$, hence all $(-1)$-curves on $X$ are contracted by $f$. If $h\colon \map X.C.$ is the induced map, then all the rational curves on $X$ are contained in the reducible fibres of $h$, since the fibres of the map $Y\longrightarrow C$ are smooth. Hence we need to bound
$$N=\sum_{c\in C}\#\big\{\text{irreducible components of }f^{-1}(c)\big\}.$$
We have
\begin{align*}
A_2:&=\sum_{c\in C}\Big(\#\big\{\text{irreducible components of }f^{-1}(c)\big\}-1\Big)
\\&=\rho(X/Y)=c_1(Y)^2-c_1(X)^2=\frac {2c_2(X)-c_1(X)^2}{3},
\end{align*}
and since $N\leq 2A_2$, we set $A=\max\{A_1,2A_2\}$. 
\end{proof}

\begin{theorem}\label{t_surface}
For any set of integers $p,k,C_1$ and $C_2$,  there exists a constant $M=M(C_1,C_2,p,k)$ such that if 
\begin{enumerate}
\item[(i)] $X$ is a smooth non-rational projective surface with $C_1=c_1(X)^2$ and $C_2=c_2(X)$,
\item[(ii)] $S_1,\dots,S_p$ are distinct prime divisors on $X$,
\end{enumerate}
then the number of log terminal models of $(X,B)$ with $B\in \mathcal E(V)$, where $V=\sum_{i=1}^p\mathbb RS_i\subseteq \Div_{\mathbb R}(X)$, is bounded by $M$. 
\end{theorem}
\begin{proof}
If $f\colon \map X.Y.$ is the log terminal model of $(X,B)$ for some $B\in \mathcal E(V)$, then it is uniquely determined by its exceptional divisors. By Lemma \ref{l_hg}, there exists a constant $M'=M'(C_1,C_2)$ such that the number of $(-1)$-curves on $X$ is bounded by $M'$. Hence, setting $M=2^{p+M'}$, the result follows by Lemma \ref{l_fact}.
\end{proof}
\section{Minimal models of threefolds}\label{sec:threefolds}

\begin{lemma}\label{l_corner}
Let $(X,S=\sum_{i=1}^p S_i)$ be a  log smooth projective threefold, where $S_1,\dots,S_p$ are distinct prime divisors, and assume that $0<\varepsilon\leq1/2$ is a rational number such that $(X,\varepsilon S)$ is terminal and $K_X+\varepsilon S$ is big.  Assume that $S_i\nsubseteq\Bs_+(K_X+\varepsilon S)$ for every $i$. Let $I$ be the total number of irreducible components of intersections of each two of the divisors $S_1,\dots,S_p$. 

Then for any $i$, the number of curves contained in 
$$\Bs_+(K_X+\varepsilon S)\cap S_i$$
is bounded  by a constant which depends on $\rho(X)$, $\rho(S_i)$, $\varepsilon$ and $I$. 
\end{lemma}
\begin{proof}
Fix an index $i$. Then there exists a sequence of $(K_X+\varepsilon S)$-flips and divisorial contractions
\begin{equation}\label{eq:sequence}
f\colon X=\rmap X^0.\rmap .\dots.X^k.\longrightarrow X^{k+1}
\end{equation}
such that $X^k$ is a log terminal model of $(X,\varepsilon S)$ and $X^{k+1}$ is the ample model $(X,\varepsilon S)$. Since $S_i\nsubseteq\Bs_+(K_X+\varepsilon S)$, the divisor $S_i$ is not contracted by this MMP by Lemma \ref{l_bplus}. Let $S^j_\ell$ and $\overline{S}_\ell^j$ denote the proper transform of $S_\ell$ in $X^j$ and its normalisation for every $\ell=1,\dots,p$, and set $S^j=\sum_{\ell=1}^p S_\ell^j$. Thus, there are induced sequences 
$$g\colon S_i=\rmap S_i^0.\rmap S_i^1..\rmap \dots.. S_i^k\dashrightarrow S_i^{k+1}.$$ 
and
$$\overline g\colon S_i=\rmap \overline S_i^0.\rmap \overline S_i^1..\rmap \dots.. \overline S_i^k\dashrightarrow \overline{S}_i^{k+1}..$$ 
By Lemma \ref{l_bplus}, if $C$ is a curve contained in $\Bs_+(K_X+\varepsilon S)\cap S_i$, then $C\subseteq\Exc(f)$. 

We first assume that $g$ is an isomorphism at the generic point of $C$. Then there exists an $f$-exceptional prime divisor $E\subseteq X$ containing $C$ such that $f(C)=f(E)\subseteq X^{k+1}$; otherwise, the exceptional set of $f$ would be $1$-dimensional in a neighbourhood of $C$, hence $g$ would not be an isomorphism at the generic point of $C$. By Lemma \ref{l_canonical} and \ref{lem:discrepbound}, we have
$$0 \leq a(E,X^{k+1},\varepsilon S^{k+1}) \leq\rho(X)- \varepsilon\mult_E S^{k+1} \le \rho(X)- \varepsilon\mult_{f(E)} S^{k+1},$$
and in particular
$$\mult_{f(E)} S_i^{k+1} < \rho(X)/\varepsilon.$$
Therefore, for each $f$-exceptional divisor $E$,  there are at most $\rho(X)/\varepsilon$ curves in $E\cap S_i$ which map to $f(E)$. Since there are at most $\rho(X/X^{k+1})$ such divisors $E$, the number of curves $C\subseteq\Bs_+(K_X+\varepsilon S)\cap S_i$ which are not contracted by $g$ is at most $\rho(X)^2/\varepsilon$. 

It remains to count the curves $C\subseteq \Bs_+(K_X+\varepsilon S)\cap S_i$ such that $g$ is not an isomorphism at the generic point of $C$, and it suffices to count the curves contracted by each of the maps $g_j\colon\rmap S_i^j.S_i^{j+1}.$. Let $\overline g_j\colon\rmap \overline S_i^j.\overline S_i^{j+1}.$ be the induced maps of normalisations, and let $N_j$ is the number of curves \emph{extracted} by $\overline g_j$. First note that for each curve contracted by $g_j$ there exists at least one curve contracted by $\overline g_j$. Thus, there are at most $\rho(\overline S_i^j)-\rho(\overline S_i^{j+1})+N_j$ curves contracted by $g_j$, and we need to bound the number $\rho(S_i)+\sum_{j=0}^k N_j$. 

If $N_j\neq 0$, then $\rmap X^j. X^{j+1}.$ must be a flip (hence necessarily $j<k$), and furthermore, $N_j$ is the number of flipped curves contained in $S_i^{j+1}$. For each such a curve $\Gamma$, let $E_\Gamma$ be the exceptional divisor obtained by blowing up $\Gamma$ which dominates $\Gamma$. Then, by Lemma \ref{l_canonical}, $X^{j+1}$ is terminal and therefore it is smooth at the generic point of $\Gamma$ by \cite[Corollary 5.39]{KM98}. Thus,
\begin{equation}\label{eq:discrep}
0\leq a(E_\Gamma,X,\varepsilon S)<a(E_\Gamma,X^{j+1},\varepsilon S^{j+1})=1-\varepsilon\mult_{\Gamma} S^{j+1}\leq 1-\varepsilon,
\end{equation}
where the last inequality follows from $\mult_\Gamma S_i^{j+1}\geq1$.

Let $\mathcal V$ be the set of all valuations which are either $f$-exceptional prime divisors on $X$, or obtained as the exceptional divisor on the blow-up of a curve in $S_\ell\cap S_i$ for each $\ell\neq i$; then it is clear that $\#\mathcal V\leq\rho(X)+I$. Viewing each $E_\Gamma$ as a valuation, we first claim that $E_\Gamma\in\mathcal V$ for all $\Gamma$. Indeed, assume that the centre of $E_\Gamma$ on $X$ is a point $x\in X$. If $E_\Gamma$ is obtained by blowing up $x$, then as $(X,S)$ is log smooth, we have
$$a(E_\Gamma,X,\varepsilon S)=2-\varepsilon\mult_x S\geq 2-3\varepsilon\geq1-\varepsilon,$$
which is a contradiction with \eqref{eq:discrep}. The case when $E_\Gamma$ is obtained by blowing up a point on a birational model of $X$ also follows since the discrepancies increase by blowing up, as $(X,\varepsilon S)$ is terminal. Therefore, the centre of $E_\Gamma$ on $X$ is either a divisor or a curve, and then the rest of the claim follows by analogous computations. In particular, we have $N_j\leq\#\mathcal V\leq\rho(X)+I$ for each $j$. 

Next we want to estimate how many times it happens that $N_j\neq0$. In other words, we want to find an upper bound on the number of varieties $X^{j+1}$ on which a valuation in $\mathcal V$ is realised as the exceptional divisor of a blow-up of a flipped curve on $X^{j+1}$. Fix $E\in\mathcal V$, and consider the number 
$$M_E^{j+1}=\mult_E S^{j+1}\in\mathbb N.$$
If $E$ is realised as the exceptional divisor on the blow-up of a flipped curve on $X^{j+1}$, then 
$$0\leq a(E,X^{j+1},\varepsilon S^{j+1})=1-\varepsilon M_E^{j+1},$$
and hence $M_E^{j+1}\leq 1/\varepsilon$ for all $j$. Since at each step of \eqref{eq:sequence} the discrepancies are increasing, the sequence $M_E^{j+1}$ is decreasing. Therefore, each $E\in\mathcal V$ is realised as an exceptional divisor on the blow-up of a flipped curve at most $1/\varepsilon$ times, hence 
$$\sum_{j=0}^kN_j\leq\frac{\rho(X)+I}{\varepsilon}.$$
Putting all this together, we get that the number of curves contained in $\Bs_+(K_X+\varepsilon S)\cap S_i$ is at most
$$\rho(S_i)+\frac {\rho(X)^2+\rho(X)+I} {\varepsilon},$$
which proves the lemma.
\end{proof}

\begin{lemma} \label{l_valuation} 
Let $(X,S=\sum_{i=1}^p S_i)$ be a log smooth projective threefold, where $S_1,\dots S_p$ are distinct prime divisors, and let $V=\sum_{i=1}^p \mathbb R_+ S_i\subseteq \Div_{\mathbb R}(X)$. Assume that $S_j\nsubseteq\Bs_+(K_X+B)$ for all $B\in \mathcal L (V)$ such that $K_X+B$ is big and for all $j$. Let $I$ be the total number of irreducible components of intersections of each two of the divisors $S_1,\dots,S_p$. 

Then for any $j$, and for every rational number $\varepsilon>0$ such that $(X,\varepsilon S)$ is terminal and $K_X+\varepsilon S$ is big, the number of curves contained in 
$$\bigcup_{B\in \mathcal L_\varepsilon (V)}\Bs_+(K_X+B)\cap S_j$$ 
is bounded  by a constant which depends on $\rho(X)$, $\rho(S_j)$, $p$, $\varepsilon$ and $I$. 
\end{lemma}

\begin{proof} 
By Lemma \ref{l_corner} there exists a constant $M=M(\varepsilon, I,\rho(X),\rho(S_j))$ such that the number of curves in $\Bs_+(K_X+\varepsilon S)\cap S_j$ is bounded by $M$.

Without loss of generality, we may assume that $\varepsilon<1/2$. Let
$$\mathcal L'(V)=\{ B=\sum a_i S_i\mid a_i\in [\varepsilon, 1]\},$$
and let $B_1,\dots, B_{2^p}$ be the extreme points of $\mathcal L'(V)$. Since $\mathcal L_\varepsilon (V)\subseteq \mathcal L'(V)$, it follows that  
$$\bigcup_{B\in \mathcal L_\varepsilon (V)}\Bs_+(K_X+B)\subseteq \bigcup_{i=1}^{2^p} \Bs_+(K_X+B_i).$$
Hence, it is enough to bound the number of curves in $\Bs_+(K_X+B_i)\cap S_j$ for every $i=1,\dots,2^p$. 
Fix $i$, and note that $\mult_{S_j}B_i\in\{\varepsilon,1\}$. We distinguish two cases. 
 
If $\mult_{S_j}B_i=1$, set $T=\varepsilon \sum_{k\neq j} S_k+S_j$. Then 
$(K_X+T)|_{S_j}$ is terminal, and let $f\colon \rmap X.X'.$ be the ample model of $K_X+T$. By assumption and by Lemma \ref{l_bplus}, $f$ does not contract $S_j$ and by Lemmas \ref{lem:BCHM} and \ref{lem:canonical}, the MMP for $(X,T)$ induces an MMP for some terminal pair $(S_j,\Theta)$. In particular, since $S_j$ is a surface, this induced MMP contracts at most $\rho(S_j)$ curves. Further, if a curve $C\subseteq \Bs_+(K_X+T)\cap S_j$ is not contracted by the MMP for $(S,\Theta)$, then similarly as in Lemma \ref{l_corner}, the number of such curves $C$ is bounded by $\frac1\varepsilon\rho(X)(\rho(X)+1)$. Thus, the number of curves inside $\Bs_+(K_X+T)\cap S_j$ is at most $\rho(S_j)+\frac1\varepsilon\rho(X)(\rho(X)+1)$. We have
\begin{align*}
\Bs_+(K_X+B_i)\cap S_j& \subseteq (\Bs_+(K_X+T)\cup \Supp (B_i-T))\cap S_j,\\
& \subseteq \Big(\Bs_+(K_X+T)\cup\bigcup\nolimits_{k\neq j}S_k\Big)\cap S_j,
\end{align*}
and hence the number of curves inside $\Bs_+(K_X+B_i)\cap S_j$ is at most $\rho(S_j)+\frac1\varepsilon\rho(X)(\rho(X)+1)+I$. 

Finally, if $\mult_{S_j} B_i=\varepsilon$, then, since $B_i\ge \varepsilon S$, we have 
\begin{align*}
\Bs_+(K_X+B_i)\cap S_j & \subseteq \big(\Bs_+(K_X+\varepsilon S)\cup \Bs_+(B_i-\varepsilon S)\big)\cap S_j\\
& \subseteq \Big(\Bs_+(K_X+\varepsilon S)\cup\bigcup\nolimits_{k\neq j}S_k\Big)\cap S_j.
\end{align*}
Thus, the number of curves in $\Bs_+(K_X+B_i)\cap  S_j$ is bounded by $M+I$ and the result follows.
\end{proof}

\begin{lemma}\label{lemma:polytopesBase}
Let $(X,\sum_{i=1}^p S_i)$ be a $3$-dimensional log smooth pair such that $K_X$ is pseudoeffective,  $S_1,\dots,S_p$ are distinct prime divisor, and let $V=\sum_{i=1}^p \mathbb R_+ S_i\subseteq \Div_{\mathbb R}(X)$. Assume that $S_i\nsubseteq\Bs(K_X+B)$ for all $B\in \mathcal L_\varepsilon (V)$ and every $i=1,\dots,p$. Let $F_1,\dots,F_\ell$ be all the prime divisors contained in $\Bs(K_X)$, and for every $\nu\subseteq\{1,\dots,\ell\}$, define 
$$\mathcal B_\nu=\{B\in\mathcal L_\varepsilon^{\can}(V)\mid F_i\subseteq\Bs(K_X+B)\text{ if and only if }i\in\nu\}.$$
Let $\mathcal C_i$ be the terminal chambers in $V$ (cf.\ Definition \ref{def:chambers}), for $1\leq i\leq k$. Assume that each adjacent-connected component of every $\mathcal B_\nu$ with respect to the covering by $\mathcal C_i$ is the union of at most $m$ polytopes $\mathcal C_i$. 

Then there exists a constant $M=M(\ell,m)$ such that $k\le M$. 
\end{lemma}
\begin{proof}
By assumptions, for any $B\in \mathcal L_\varepsilon^{\can} (V)$, any prime divisor in $\Bs(K_X+B)$ must be one of $F_j$, and for each $1\leq i\leq\ell$ denote 
$$\mathcal P_i=\{B\in \mathcal L_\varepsilon (V)\mid F_i\nsubseteq\Bs(K_X+B)\}.$$
Then for any $\nu\subsetneq \{1,\dots,\ell\}$, the set $\mathcal B_\nu$ is the closure of $\bigcup_{i\notin \nu} \mathcal P_i \backslash \bigcup_{j\in \nu}\mathcal P_j$, and $\mathcal B_{\{1,\dots,\ell\}}$ is the closure of $\mathcal L_\varepsilon^{\can} (V)\setminus \bigcup_{i=1}^\ell \mathcal P_i$. It is clear that every $\mathcal P_i$ satisfies the relation \eqref{eq:polytopes} on page \pageref{eq:polytopes}, and we conclude by Lemma \ref{lemma:polytopes}.
\end{proof}

\begin{theorem}\label{l_disconnected}
Let $p$ and $\rho$ be positive integers, and let $\varepsilon$ be a positive rational number. Let $(X,\sum_{i=1}^p S_i)$ be a $3$-dimensional log smooth pair such that:
\begin{enumerate}
\item[(i)] $K_X$ is pseudoeffective;
\item[(ii)] $S_1,\dots,S_p$ are distinct prime divisor which are not contained in $\Bs(K_X+B)$ for all $B\in \mathcal L (V)$, 
\item[(iii)] the vector space $V=\sum_{i=1}^p\mathbb R S_i\subseteq \Div_{\mathbb R}(X)$ spans $\Div_{\mathbb R}(X)$ up to numerical equivalence,
\item[(iv)] $\rho(X)\le \rho$ and $\rho(S_i)\le \rho$ for all $i=1,\dots,p$.
\end{enumerate}
Let $I$ be the total number of irreducible components of intersections of each two and each three of the divisors $S_1,\dots,S_p$. 

Then there exists a constant $N=N(p,\rho,\varepsilon,I)$ such that the number of terminal chambers in $V$ which intersect the interior of $\mathcal L_\varepsilon(V)$ is at most $N$. 
\end{theorem}

\begin{proof}
Let $(X,\sum_{i=1}^p S_i)$ be a 3-dimensional log smooth pair satisfying the conditions (i)--(iv). Note that $K_X+B$ is big for every $B\in\mathcal L_\varepsilon(V)$. Let $C_1,\dots, C_q$ be all the curves contained in 
$$\bigcup_{B\in \mathcal L_\varepsilon (V)}\Bs_+(K_X+B)\cap S.$$
Then $q$ is bounded by a constant depending on $p$, $\rho$, $\varepsilon$ and $I$ by Lemma \ref{l_valuation}. By Lemma \ref{l_echo}, there are finitely many geometric valuations $E_1,\dots,E_m$ such that $c_X(E_j)\subseteq\bigcup_{i=1}^q C_i$ for all $j$ and $a(E_j,X,B)<1$ for some $B\in \mathcal L_\varepsilon (V)$, and $m\leq M=M(q,\rho,\varepsilon,I)$. 

Let $F_1,\dots,F_\ell$ be all the prime divisors in $\Bs(K_X)$. Then by (ii), for every $B\in\mathcal L_\varepsilon(V)$ the divisorial part of $\Bs(K_X+B)$ is contained in $\sum F_i$. Let $f=f_B\colon \rmap X.X_B.$ be a log terminal model of $(X,B)$. For every $\nu\subseteq\{1,\dots,\ell\}$, let 
$$\mathcal B_\nu=\{B\in\mathcal L_\varepsilon(V)\mid F_i\text{ is contracted by }f_B\text{ if and only if }i\in\nu\}.$$
Then by Lemma \ref{lemma:polytopesBase}, it is enough to bound the number of terminal chambers which intersect each adjacent-connected component of each $\mathcal B_\nu$. 

Hence, from now on we fix such $\nu$ and we assume, as we may, that each $\mathcal B_\nu$ is adjacent-connected. We will show that the number of terminal chambers which intersect $B_\nu$ is bounded by a constant depending only on $p$, $\rho$ and $\varepsilon$, which is enough to conclude.

Set $\mu=\rho+M$; then $\mu$ depends only on $p,\rho$, $\varepsilon$ and $I$ by above. Let $\mathcal{S}$ be the set of all $p$-tuples $(m_1,\dots,m_p)\in\mathbb N^p$ such that $m_i<\mu/\varepsilon$ for every $i$. Then $\#\mathcal S<(\mu/\varepsilon)^p$. Let $\mathcal H$ be the set of all hyperplanes $\langle\Sigma_1-\Sigma_2,\mathbf x\rangle=r$, where $\Sigma_1\neq \Sigma_2$ are elements of $\mathcal S$ and $-\mu<r<\mu$ is an integer. Then $\#\mathcal H\leq2\mu\binom{(\mu/\varepsilon)^p}{2}$. The elements of $\mathcal H$ subdivide $\mathcal B_\nu$ into at most $2^{\#\mathcal H}$ polytopes, and by replacing $\mathcal B_\nu$ by any of these polytopes, we may assume that none of the elements of $\mathcal H$ intersects the interior of $\mathcal{B}_\nu$. It is now enough to show that there is exactly one terminal chamber whose interior intersects $\mathcal B_\nu$.

Assume that there are two adjacent terminal chambers $\mathcal C'$ and $\mathcal C''$ whose interiors intersect $\mathcal B_\nu$. Let $X'$ and $X''$ be the corresponding log terminal models, let $B=\sum_{i=1}^p b_i S_i$ be a divisor in $\mathcal C''$, and let $B'$ and $S_i'$, respectively $B''$ and $S_i''$ be the proper transforms of $B$ and $S_i$ on $X'$ and $X''$. Note that $X'$ and $X''$ are terminal by Lemma \ref{l_canonical}. Denote $\mathbf b=(b_1,\dots,b_p)$ and let $\langle\, ,\rangle$ denote the standard scalar product on $V$. For each geometric valuation $E$ on $X$, define 
$$\Sigma_{E,\mathcal C'}= (\mult_E S_1',\dots,\mult_E S_p'),\quad \Sigma_{E,\mathcal C''}= (\mult_E S_1'',\dots,\mult_E S_p'').$$

By the definition of $\mathcal B_\nu$, and possibly by relabelling the chambers, we may assume that the induced map $\rmap X'.X''.$ is the flip of $(X',B')$. Note that $X'$ is the ample model of $(X,\Delta)$ for any $\Delta$ in the interior of $\mathcal C'$, and similarly for $\mathcal C''$. Let $C\subseteq X''$ be a flipped curve, and let $E$ be the valuation on $X''$ obtained by blowing up $C$ which dominates $C$. Since $X''$ is smooth at the generic point of $C$ by \cite[Corollary 5.39]{KM98}, we have
\begin{equation}\label{eq:a2}
0<a(E,X,B)<a(E,X'',B'')=1-\langle\Sigma_{E,\mathcal{C}''},\mathbf b\rangle\leq1.
\end{equation}
It is easy to see from the discrepancy formulas that then $c_X(E)$ belongs to some of the divisors $S_1,\dots,S_p$ since $(X,B)$ is terminal and $a(E,X,B)<1$. Moreover, if $B$ belongs to the interior of $\mathcal C''$, then $X''=\Proj R(X,K_X+B)$, and hence $c_X(E)\subseteq \Bs_+(K_X+B)$ by Lemma \ref{l_bplus}. This shows that $E$ is one of the valuations $E_1,\dots,E_m$. 

Furthermore, by Lemma \ref{lem:discrepbound} we have 
\begin{equation}\label{eq:a1}
0<a(E,X',B')=\mu_{E,B}-\langle\Sigma_{E,\mathcal{C}'},\mathbf b\rangle
\end{equation}
for some integer $0<\mu_{E,B}<\mu$. Since $b_i\geq\varepsilon$ for all $i$, we have $0\leq\mult_{E} S_i'<\mu/\varepsilon$ for all $i$, and in particular, $\Sigma_{E,\mathcal C'}\in\mathcal S$. 

Now, if $B\in\mathcal C'\cap \mathcal C''$, then by the Negativity lemma we have $a(E,X',B')=a(E,X'',B'')$. Together with \eqref{eq:a2}, \eqref{eq:a1} and the fact that none of the elements $\mathcal H$ intersects the interior of $\mathcal B_\nu$, this implies that $\Sigma_{E,\mathcal C'}=\Sigma_{E,\mathcal C''}$. On the other hand, if $B$ belongs to the interior of $\mathcal C''$, then the Negativity lemma again gives
$$a(E,X',B')< a(E,X'',B''),$$ 
and this together with \eqref{eq:a2} and \eqref{eq:a1} implies $\mu_{E,B}<1$, a contradiction. 
\end{proof}

We are now ready to give proofs of our main results.

\begin{proof}[Proof of Theorem \ref{t_main}]
The number of terminal chambers inside of the set 
$$\{\sum a_iS_i\mid a_i\in [\varepsilon/2,1-\varepsilon/2]\}$$
is bounded by a constant $N=N(p,\rho,\varepsilon/2)$ by Theorem \ref{l_disconnected}. We set $C=N$.
\end{proof}

\begin{proof}[Proof of Corollary \ref{corollary}]
It is clear that the total number of irreducible components of intersections of each two and each three of the components of $\Delta_0$ and $\Delta$ is the same under a homeomorphism which preserves the topological type of $(X,\Delta_0)$. Therefore, the result follows immediately from Theorem \ref{l_disconnected}.
\end{proof}

\bibliographystyle{amsalpha}

\bibliography{Library}

\end{document}